\documentclass[11pt]{article}
\usepackage{a4wide}
\usepackage{amsmath,amsthm,amssymb}
\usepackage{units}
\usepackage{url}
\usepackage{calc, graphicx} 
\usepackage{bbm}
\usepackage{dsfont}
\newcommand{\ind}{\ensuremath{\mathds{1}}}

\newtheorem{thm}{Theorem}[section]
\newtheorem{lem}[thm]{Lemma}
\newtheorem{prop}[thm]{Proposition}

{\theoremstyle{definition}
  
}

\newcommand{\C}{\ensuremath{\mathbb{C}}}
\newcommand{\R}{\ensuremath{\mathbb{R}}}
\newcommand{\E}{\ensuremath{\mathbb{E}}}

\newcommand{\Cov}{\ensuremath{\mathrm{Cov}}}

\newsavebox{\smlmat}
\savebox{\smlmat}{$\left[\begin{smallmatrix}1 & 2 \\ 2 &1\end{smallmatrix} \right]$ }

\title{The CLT Analogue for Cyclic Urns}
\author{Noela S.~M\"uller and Ralph Neininger\\
Institute for Mathematics\\
J.W.~Goethe University\\
60054 Frankfurt a.M.\\
Germany\\ \\
Email: {\tt \{nmueller,neiningr\}@math.uni-frankfurt.de}}
\begin{document}

\maketitle

\begin{abstract}
A cyclic urn is an urn model for balls of types $0,\ldots,m-1$ where in each draw the ball drawn, say of type $j$, is returned to the urn together with a new ball of type $j+1 \mod m$. The case $m=2$ is the well-known Friedman urn. The composition vector, i.e., the vector of the numbers of balls of each type after $n$ steps is, after normalization, known to be asymptotically normal for $2\le m\le 6$. For $m\ge 7$ the normalized composition vector does not converge. However, there is an almost sure approximation by a periodic random vector. In this paper  the asymptotic fluctuations around this periodic random vector are identified.  We show that these fluctuations are asymptotically normal for all $m\ge 7$. However, they are of maximal dimension $m-1$ only when $6$ does not divide $m$. For $m$ being a multiple of $6$ the fluctuations are supported by a two-dimensional subspace.  \end{abstract}

\noindent
{\textbf{MSC2010:} 60F05, 60F15, 60C05, 60J10.


\section{Introduction, phenomena and results}
The aim of this extended abstract is to uncover the nature of fluctuations around almost surely oscillating sequences of random variables as they arise in a number of random combinatorial structures, most commonly in random trees. We develop an analysis for the composition vector of cyclic urns and describe at this example the new phenomena and characteristics of the fine fluctuations around a random oscillating sequence which (in an almost sure sense) approximates the normalized composition vector of a cyclic urn.

A cyclic urn is an urn model with a fixed number $m\ge 2$ of possible colours of balls which we call types $0,\ldots,m-1$. Initially, there is one ball of an arbitrary type. In each step we draw a ball from the urn, uniformly from within the balls in the urn and independently of the history of the urn process. If its type is $j\in\{0,\ldots,m-1\}$ it is placed back to the urn together with a new ball of type $j+1 \mod m$. We denote by $R_n=(R_{n,0},\ldots,R_{n,m-1})^t$ the (column) vector of the numbers of balls of each type after $n$ steps when starting with one ball of type $0$. Hence, we have $R_0=e_0$ where $e_j$ denotes the $j$-th unit vector in $\R^m$, indexing the unit vectors by $0,\ldots,m-1$. For fixed $m\ge 2$ we denote the $m$-th elementary root of unity by $\omega:=\exp(\frac{2\pi\mathrm{i}}{m})$. Furthermore we set
\begin{align}
&\lambda_k:= \Re(\omega^k)=\cos\left(\frac{2\pi k}{m}\right),\qquad \mu_k:= \Im(\omega^k)=\sin\left(\frac{2\pi k}{m}\right),\nonumber\\
&v_k:=\frac{1}{m}\left(1,\omega^{-k},\omega^{-2k},\ldots,
\omega^{-(m-1)k}\right)^t\in\C^m,\quad 0\le k\le m-1. \label{eig_vek}
\end{align}
Note that $v_0= \frac{1}{m}\mathbf{1}:=\frac{1}{m}(1,1,\ldots,1)^t\in \R^m$.

The asymptotic distributional behavior of the sequence $(R_n)_{n\ge 0}$ has  been identified in Janson \cite{Ja83,Ja04,Ja06}, see also Pouyanne \cite{Pou05,Pou08}. Janson also developed a limit theory for the compositions of rather general urn schemes. For simplicity of presentation we state the case when starting with one ball of type $0$. However, when starting with one ball of type $j\in\{0,\ldots,m-1\}$, the corresponding composition vector $R_n^{[j]}$ is obtained in distribution by the relation
 \begin{align}\label{shift}
R_n^{[j]} \stackrel{d}{=}\left({\cal R}^t\right)^j R_n,\quad 0\le j\le m-1,
\end{align}
where the replacement matrix $\mathcal{R}$ is defined in (\ref{R}). Hence, it is sufficient to consider the cyclic urn process started with one ball of colour $0$. An extension to initially having more than one ball is straightforward, see the discussion in \cite[p.~1165]{knne14}.

For the cyclic urns  Janson showed that for $2\le m\le 6$ the normalized composition vector $R_n$ converges in distribution towards a multivariate normal distribution, whereas for $m\ge 7$ there is no convergence by a conventionally standardized version of  the $R_n$ due to subtle periodicities. For $m\ge 7$ there exists  a complex valued random variable $\Xi_1$ (depending on $m$) such that almost surely, as $n\to\infty$, we have
\begin{align}\label{strong}
\frac{R_n- \frac{n}{m}\mathbf{1}}{n^{\lambda_1}}
-2\Re\left(n^{i\mu_1}\Xi_1v_1\right) \to 0.
\end{align}

We now focus on the periodic case $m\ge 7$.  According to
(\ref{strong}) the normalization
$n^{-\lambda_1}(R_n- \frac{n}{m}\mathbf{1})$ does not
converge but is (strongly) approximated by the oscillating
random sequence $(2\Re(n^{i\mu_1}\Xi_1v_1))_{n\ge 0}$. In
the present paper we clarify whether  it is still possible
that the fluctuations of  the $n^{-\lambda_1}(R_n- \frac{n}{m}\mathbf{1})$ around the periodic sequence $(2\Re(n^{i\mu_1}\Xi_1v_1))_{n\ge 0}$  do converge although the sequence itself does not converge. Subsequently, we will call the differences in (\ref{strong})  residuals.

Our main results stated in Theorems \ref{thm1} and \ref{thm2} show that the nature of the asymptotic behavior of the  residuals in (\ref{strong}) depends on the number of colours $m$. For $m\in\{7,8,9,10,11\}$ there is a direct normalization which implies a  multivariate central limit law (CLT) for the residuals. The case $m=12$ also allows a multivariate CLT with a different scaling. For $m>12$ the residuals cannot directly by normalized to obtain convergence. However, considering refined residuals allows a multivariate CLT for all $m>12$. This in fact gives a more refined expansion of the $R_n$, cf.~Theorems \ref{thm1} and \ref{thm2}. There is a further subtlety in the nature of the fluctuations of the residuals: If $6$ divides $m$ the fluctuations of the residuals are asymptotically supported by a two-dimensional plane, i.e., the covariance matrix of the limit normal distribution has rank 2, whereas for all  $m\ge 7$ which are not divided by $6$ this support is a hyperplane (rank $m-1$).

By $\stackrel{d}{\longrightarrow}$ (and $\stackrel{d}{=}$) convergence (resp.~equality) in distribution are denoted,
 for a symmetric positive semi-definite matrix $M$ by ${\cal N}(0,M)$ the centered normal distribution with covariance matrix $M$. For $v\in\C^m$ we denote by $v^*$ the conjugate transpose of $v$. Furthermore, $6\mid m$ and $6 \nmid m$ is short for $6$  divides (resp.~does not divide) $m$.

We distinguish the cases $6\mid m$ and $6 \nmid m$ as follows:
\begin{thm}\label{thm1}
Let $m \geq 7$ with $6\nmid m$ and set $r:=\lfloor (m-1)/6\rfloor$. Then, there exist complex valued random variables $\Xi_1,\ldots,\Xi_r$ such that, as $n \to \infty$, we have
\begin{equation*}
n^{\lambda_1-1/2}\left(\frac{R_n- \E[R_n]}{n^{\lambda_1}} - \sum_{k=1}^r 2n^{\lambda_k-\lambda_1}
\Re\left(n^{i\mu_k}\Xi_k v_k\right)\right) \stackrel{d}{\longrightarrow}\mathcal{N}\left(0,\Sigma^{(m)}\right).
\end{equation*}
The covariance matrix $\Sigma^{(m)}$ has rank $m-1$ and is given by
\begin{equation*}
\Sigma^{(m)}= \sum_{k=1}^{m-1}\frac{1}{|2\lambda_k-1|} v_k v_k^*.
\end{equation*}
\end{thm}
When $6\mid m$ the normalization requires an additional $\sqrt{\log n}$ factor and the rank of the covariance matrix is reduced to $2$:
\begin{thm}\label{thm2}
Let $m \geq 7$ with $6 \mid m$ and set $r:=\lfloor (m-1)/6\rfloor$. Then, there exist complex valued random variables $\Xi_1,\ldots,\Xi_r$  such that, as $n \to \infty$, we have
\begin{equation*}
\frac{n^{\lambda_1-1/2}}{\sqrt{\log(n)}}\left(\frac{R_n- \E[R_n]}{n^{\lambda_1}} - \sum_{k=1}^r 2n^{\lambda_k-\lambda_1}
\Re\left(n^{i\mu_k}\Xi_k v_k\right)\right) \stackrel{d}{\longrightarrow}\mathcal{N}\left(0,\Sigma^{(m)}\right).
\end{equation*}
The covariance matrix $\Sigma^{(m)}$ has rank $2$ and is given by
\begin{equation*}
\Sigma^{(m)}= v_{m/6}v_{m/6}^*+ v_{5m/6}v_{5m/6}^*.
\end{equation*}
 \end{thm}

The convergences  in Theorems \ref{thm1} and \ref{thm2} also hold with all moments. For an expansion of $\E[R_n]$ see (\ref{exp_mean}).

We consider Theorems \ref{thm1} and \ref{thm2} as prototypical for a phenomenon which we conjecture to occur frequently in related random combinatorial structures. E.g., we expect similar behavior for the size of random $m$-ary search trees, cf.~\cite{chhw01,chpo04,FiKa04}, and for the number of leaves in random $d$-dimensional (point) quadtrees \cite{chfuhw07}. (For both instances only the case of Theorem \ref{thm1} is expected to occur.)

\section{Outline of the proof}
In this section we first recall some known asymptotic behavior of $R_n$ which is used subsequently. Then we state a more refined result on certain projections of residuals in Proposition \ref{prop1} which directly implies Theorems \ref{thm1} and \ref{thm2}. Then, an outline of the proof of Proposition \ref{prop1} is given. Technical steps and estimates are then sketched in Section \ref{sec:3}. Throughout, we fix an $m\ge 7$.

The cyclic urn with $m$ colours has the $m\times m$ replacement matrix
\begin{equation} \label{R}
{\cal R}:=
\begin{pmatrix}
0 & 1 & 0& \cdot & \cdot & 0 & 0 \\
0 & 0 & 1& \cdot & \cdot & 0 & 0 \\
0 & 0 & 0& \cdot & \cdot & \cdot & \cdot \\
\cdot & \cdot &  \cdot & \cdot & \cdot & \cdot & \cdot \\
\cdot & \cdot & \cdot & \cdot & \cdot & 0 & 1 \\
1 & 0 & 0& \cdot & \cdot & 0 & 0
\end{pmatrix},
\end{equation}
where ${\cal R}_{ij}$ indicates that after drawing a ball of type $i$ it is placed back together with ${\cal R}_{ij}$ balls of type $j$ for all $0\le i,j\le m-1$.
For the urn we consider the initial configuration of one ball of type $0$ and write $R_n$ for the composition vector after $n$ steps. The canonical filtration is given by the $\sigma$-fields ${\cal F}_n=\sigma(R_0,\ldots,R_n)$ for $n\ge 0$.
The dynamics of the urn process imply that, almost surely, we have
\begin{align}\label{bed_erw}
\mathbb{E}\left[R_{n+1} \,|\, \mathcal{F}_n \right]
= \sum_{k=0}^{m-1} \frac{R_{n,k}}{n+1} (R_{n} + {\cal R}^t e_k )
= \left(\mathrm{Id}_m + \frac{1}{n+1}{\cal R}^t \right)R_n,\qquad n\ge 0.
\end{align}
Here, $\mathrm{Id}_m$ denotes the $m\times m$ identity matrix and ${\cal R}^t$ the transpose of ${\cal R}$. The matrices ${\cal R}$ and $\mathrm{Id}_m + \frac{1}{n+1}{\cal R}^t$ have the same (right) eigenvectors $v_0,\ldots,v_{m-1}$ given in (\ref{eig_vek}).

Note that $v_0$ has the direction of the drift vector $\mathbf{1}$ in Theorems \ref{thm1} and \ref{thm2} and $v_1$ determines  the directions of the a.s.~fluctuations around the drift there. By diagonalizing these matrices and using  (\ref{bed_erw}) one finds explicit expressions for the mean of the $R_n$, cf.~\cite[Lemma 6.7]{knne14}.
With
\begin{equation*}
\xi_k := \frac{2}{\Gamma(1+\omega^k)}v_k,\quad 1\le k \le r,
\end{equation*}
these expressions imply the expansion, as $n\to\infty$,
\begin{equation}\label{exp_mean}
\mathbb{E}\left[R_n\right] = \frac{n+1}{m}\mathbf{1} +\sum_{k=1}^r \Re(n^{i\mu_k}\xi_k)n^{\lambda_k} +   \mathrm{O}(\sqrt{n}).
\end{equation}
It is also known  that the variances and covariances of $R_n$ are of the order $n^{2\lambda_1}$ with  appropriate periodic prefactors. This explains the normalization $n^{-\lambda_1}(R_n-\frac{n+1}{m}\mathbf{1})$ in Theorems \ref{thm1} and \ref{thm2}.
The analysis of the asymptotic distribution as stated in (\ref{strong}) has been done by different techniques (partly only in a weak sense), by embedding into continuous time multitype branching processes, by (more direct) use of martingale arguments, and by stochastic fixed-point arguments, see \cite{Ja04,Pou05,knne14}.

For our further analysis we use a spectral decomposition of the process $(R_n)_{n\ge 0}$. We denote by $\pi_k$ the projection onto the eigenspace in $\C^m$ spanned by $v_k$ for $0\le k\le m-1$. Hence, we have
\begin{align*}
R_n=\sum_{k=0}^{m-1} \pi_k(R_n)=\pi_0(R_n)+\sum_{k=1}^{\lfloor m/2\rfloor} (\pi_{k}+\pi_{m-k})(R_n) + \ind_{\{ m \mbox{ even}\}}\pi_{m/2}(R_n),
\end{align*}
where $\ind$ indicates an indicator. We have deterministically $\pi_0(R_n)=\frac{n+1}{m}\mathbf{1}$. For the other projections $\pi_k(R_n)$  one has  similar periodic behavior as for the composition vector $R_n$, cf.~(\ref{strong}), as long as we have $\lambda_k>\frac{1}{2}$. We call the projections $\pi_k(R_n)$  large, if $\lambda_k>\frac{1}{2}$, since their  magnitudes have orders larger than $\sqrt{n}$. Projections $\pi_k$ with $\lambda_k\le \frac{1}{2}$ we call small. For the large projections we have for all $1\le k\le \lfloor m/2\rfloor$ with $\lambda_k>\frac{1}{2}$
almost surely that
\begin{align} \label{grenz}
Y_{n,k}:=\frac{1}{n^{\lambda_k}}(\pi_k+\pi_{m-k})(R_n - \E[R_n]) - 2\Re\left(n^{\mathrm{i}\mu_k}\Xi_k v_k\right) \to 0
\end{align}
with a complex valued random variable $\Xi_k$. The small projections $\pi_k(R_n)$ behave differently, see \cite{Ja04,mai14}. For those $k$ with $\lambda_k<\frac{1}{2}$ we have
\begin{align}\label{rn1}
X_{n,k}:=\frac{1}{\sqrt{n}}(\pi_k+\pi_{m-k})(R_n- \E[R_n]) \stackrel{\mathrm{d}}{\longrightarrow} {\cal N}(0,\Sigma_k),
\end{align}
with an appropriate covariance matrix $\Sigma_k$, see (\ref{sig_k_def1})--(\ref{sig_k_def3}).

If $m$ is even then for $X_{n,m/2}:=n^{-1/2}\pi_{m/2}(R_n)$ we have a multivariate CLT as in (\ref{rn1}).

Finally, if $6 \mid m$, then there is the pair $(\frac{m}{6},\frac{5m}{6})$ with $\lambda_{m/6}=\lambda_{5m/6}=\frac{1}{2}$. In this case the scaling requires an additional $\sqrt{\log n}$ factor. We have
\begin{align}\label{rn2}
X_{n,m/6}:=\frac{1}{\sqrt{n\log n}}(\pi_{m/6}+\pi_{5m/6})(R_n - \E[R_n]) \stackrel{\mathrm{d}}{\longrightarrow} {\cal N}(0,\Sigma_{m/6}).
\end{align}
We identify the orders of the variances and covariances of $Y_{n,k}$ in Section \ref{sec:31}. These orders imply that an appropriate normalization to study the fluctuations of the large projections is given by
\begin{align}\label{rn3}
X_{n,k}:= n^{\lambda_k-\frac{1}{2}} Y_{n,k}.
\end{align}
Now, the $X_{n,k}$ are defined for all $1\le k\le \lfloor m/2 \rfloor$ and describe the normalized fluctuations of all the projections.
For the small projections we already know that they are asymptotically normally distributed, see (\ref{rn1}). As a main contribution of the present paper we show that the residuals of the large projections as normalized in (\ref{rn3}) are also asymptotically normal. Moreover, we show that all these fluctuations are jointly asymptotically normally distributed and asymptotically independent:

\begin{prop}\label{prop1}
For the vector $(X_{n,1}, \ldots, X_{n,\lfloor m/2 \rfloor})$ defined in
(\ref{rn1})  - (\ref{rn3}) we have
\begin{align*}
(X_{n,1},\ldots,X_{n,\lfloor m/2 \rfloor}) \stackrel{d}{\longrightarrow} {\cal N}(0,\mathrm{diag}(\Sigma_1,\ldots,\Sigma_{\lfloor m/2 \rfloor})),
\end{align*}
where the blocks $\Sigma_k$ of the diagonal block matrix $ \mathrm{diag}(\Sigma_1,\ldots,\Sigma_{\lfloor m/2 \rfloor})$ are defined in (\ref{sig_k_def1})--(\ref{sig_k_def3}).
\end{prop}

Proposition \ref{prop1} directly implies Theorems \ref{thm1} and \ref{thm2}:

\proof[Proof of Theorem \ref{thm1}]
Let $m\ge 7$ with $6 \nmid m$, set $r=\lfloor (m-1)/6\rfloor$  and  let $\Xi_1,\ldots,\Xi_r$ as in (\ref{rn1}). Moreover, $X_{n,1},\ldots,X_{n,\lfloor m/2 \rfloor}$ as in Proposition \ref{prop1}. Note that $6 \nmid m$ implies that there is no $1\le k \le m$ with $\lambda_k=\frac{1}{2}$. We obtain 
\begin{align*}
\lefteqn{n^{\lambda_1-1/2}\left(\frac{R_n- \E[R_n]}{n^{\lambda_1}} - \sum_{k=1}^r 2n^{\lambda_k-\lambda_1}
\Re\left(n^{i\mu_k}\Xi_k v_k\right)\right)}\\
&= n^{\lambda_1-1/2}\left(n^{-\lambda_1}\sum_{k=1}^r \left\{(\pi_k+\pi_{m-k})(R_n- \E[R_n])-2n^{\lambda_k}
\Re\left(n^{i\mu_k}\Xi_k v_k\right)\right\} \right.\\
&\left.\qquad\qquad\quad ~+ n^{-\lambda_1}\sum_{r+1}^{\lceil m/2 \rceil-1}(\pi_k+\pi_{m-k})(R_n- \E[R_n]) +\ind_{\{m \mbox{ even}\}} n^{-\lambda_1}\pi_{m/2}(R_n- \E[R_n])\right)\\
&= X_{n,1}+\cdots+X_{n,\lfloor m/2 \rfloor}\\
&\stackrel{\mathrm{d}}{\longrightarrow} {\cal N}\left(0,\Sigma^{(m)}\right),
\end{align*}
by Proposition \ref{prop1} and the continuous mapping theorem, where $\Sigma^{(m)}=\Sigma_1+\cdots+\Sigma_{\lfloor m/2 \rfloor}$. That $\Sigma^{(m)}$ has rank $m-1$ is proven in Theorem \ref{thm_rank}.
\qed
\proof[Proof of Theorem \ref{thm2}]
Let $m\ge 7$ with $6 \mid m$ and $\Xi_1,\ldots,\Xi_r$ as in (\ref{rn1}) and $X_{n,1},\ldots,X_{n,m/2}$ as in Proposition \ref{prop1}. Note that $6 \mid m$ implies that there is the pair $(m/6,5m/6)$ with  $\lambda_{m/6}=\lambda_{5m/6}=\frac{1}{2}$. Rearranging terms as in the proof of Theorem \ref{thm1} we obtain
\begin{align*}
\frac{n^{\lambda_1-1/2}}{\sqrt{\log n}}\left(\frac{R_n- \E[R_n]}{n^{\lambda_1}} - \sum_{k=1}^r 2n^{\lambda_k-\lambda_1}
\Re\left(n^{i\mu_k}\Xi_k v_k\right)\right)
&= X_{n,m/6} + \frac{1}{\sqrt{\log n}}\sum_{k=1\atop k\neq m/6}^{m/2}X_{n,k}\\
&\stackrel{\mathrm{d}}{\longrightarrow} {\cal N}\left(0,\Sigma^{(m)}\right),
\end{align*}
by Proposition \ref{prop1} and Slutzky's Lemma, where $\Sigma^{(m)}=\Sigma_{m/6}$.  That $\Sigma^{(m)}$ has  rank $2$ is proven in Theorem \ref{thm_rank}.
\qed\\


To prove Proposition \ref{prop1} we first derive moments and mixed moments needed for the normalization in Section \ref{sec:31}. The ranks of the covariance matrices $\Sigma^{(m)}$ are identified in Section \ref{sec:32}. In Section \ref{sec:33} a pointwise recursive equation
 for the complex random variables $\Xi_1,\ldots,\Xi_r$ is obtained together with a recurrence for the sequence $(R_n)_{n\ge 0}$ which extends to a recurrence for the residuals in (\ref{strong}) as well as to the residuals of the projections of the $R_n$. Finally, the joint convergence of the normalized residuals of all projections is finally shown by an application of a stochastic fixed-point argument in the context of the contraction method by use of the Zolotarev metric $\zeta_3$. However, only an indication and  a solid reference are given in  Section \ref{sec:34}.



\section{Sketch of the proof of Proposition \ref{prop1}} \label{sec:3}
\subsection{Proper normalization of the residuals}\label{sec:31}
Denoting the inner product in $\C^m$ by $\langle \,\cdot\, , \, \cdot \,\rangle$ we first write the spectral decomposition of the centered composition vector with respect to the orthonormal basis $\{\sqrt{m} v_k: 0 \leq k < m\}$  of the unitary vector space $\C^m$ as
\begin{align*}
 R_n - \mathbb{E}[ R_n] = \sum_{k=0}^{m-1} \pi_k\left(R_n - \E[R_n]\right)
=:\sum_{k=0}^{m-1} u_k\left(R_n - \mathbb{E}[R_n ]\right)v_k.
\end{align*}
The evolution (\ref{bed_erw}) of the process implies that the random variables
\begin{align}
M_{n,k} := \frac{\Gamma(n+1)}{\Gamma(n+1+\omega^k)} u_k\left(R_n - \mathbb{E}\left[R_n \right]\right)
\end{align}
for $k \in \{0, \ldots, m-1\} \setminus \{m/2\}$ and
\begin{align}
M_{n,m/2} :=n \cdot  u_{m/2}\left(R_n - \mathbb{E}\left[R_n \right]\right)
\end{align}
define complex-valued, centered martingales. Note, that the corresponding martingales $M_{n,k}^{[j]}$ when starting with one ball of type $j\in\{0,\ldots,m-1\}$   satisfy
\begin{equation*}
M_{n,k}^{[j+1]} \stackrel{d}{=} \omega^{k} M_{n,k}^{[j]} \qquad \mbox{(convention } M_{n,k}^{[m]}:= M_{n,k}^{[0]}\mbox{)}.
\end{equation*}
It is known, see \cite{Ja04,Ja06,Pou05}, that for all $k \in \{0, \ldots, m-1\}$ with $\lambda_k=\Re\left(\omega^k\right)>1/2$, there exists a complex random variable $\Xi_{k}$ such that, as $n \to \infty$, we have
\begin{align}\label{mg_conv}
M_{n,k} \to \Xi_k\; \mbox{ almost surely},
\end{align}
where the convergence also holds in $\mathrm{L}_p$ for every $p \geq 1$. The $M_{n,k}$ with $\lambda_k=\Re\left(\omega^k\right)\le 1/2$ are also known to converge, after proper normalization, to normal limit laws. 

Our subsequent analysis requires asymptotics for moments of and correlations between the $u_k(R_n)$. Exploiting the dynamic of the urn in (\ref{bed_erw}) elementary calculations imply that:
\begin{lem}\label{erw}
For all $k \in \{0, \ldots, m-1\} \setminus \{m/2\}$, we have
\begin{equation*}
\mathbb{E}\left[u_k\left(R_n\right)\right] = \sum_{t=0}^{m-1} \omega^{kt} \mathbb{E}\left[R_{n,t}\right] = \frac{\Gamma(n+1+\omega^k)}{\Gamma(n+1)\Gamma(1+\omega^k)},
\end{equation*}
while
\begin{equation*}
\mathbb{E}\left[u_{m/2}\left(R_n\right)\right] = 0.
\end{equation*}
For all $k, \ell \in \{0, \ldots, m-1\}$,
\begin{align*}
 \mathbb{E}\left[u_{k}\left(R_n\right)u_{\ell}\left(R_n\right)\right]
=&\prod_{s=1}^{n}\left(1+\frac{\omega^{k}+\omega^{\ell}}{s}\right) \\
&~+ \omega^{k+\ell}\sum_{s=1}^{n}\frac{1}{s}\prod_{t=1}^{s-1}\left(1+ \frac{\omega^{k+\ell}}{t} \right)\prod_{t=s+1}^{n}\left(1+\frac{\omega^{k}+\omega^{\ell}}{t}\right).
\end{align*}
\end{lem}
From Lemma \ref{erw}  we obtain the $\mathrm{L}_2$-distance of the residuals of the martingales $(M_{n,k})_{n\ge 0}$ with $\lambda_k>\frac{1}{2}$ needed for the proper normalization of these residuals:
\begin{lem} \label{mgconv}
For $k \geq 1$ such that $\lambda_k >1/2$, as $n \to \infty$, we have
\begin{equation*}
 \E\left[\left|M_{n,k} - \Xi_k\right|^2\right] \sim \frac{1}{2 \lambda_k -1}n^{1-2\lambda_k}.
\end{equation*}
\end{lem}
Lemma \ref{mgconv} directly implies the asymptotic covariances of the residuals of the centered projections of the composition vector, which we denote by
\begin{equation*}
\Pi_{n,k}:=\left\{
 \begin{array}{cl}
 \frac{\Gamma(n+1+\omega^k)}{\Gamma(n+1)}\left(M_{n,k} - \Xi_k\right)v_k, &\mbox{if } \lambda_k>\frac{1}{2},\vspace{2mm}\\
 u_k\left( R_n - \E\left[ R_n\right]\right)v_k, &\mbox{if } \lambda_k\le\frac{1}{2}.
 \end{array} \right.
\end{equation*}
Note that this notation implies the representation
\begin{equation*}
(R_n-\E[R_n]) - \sum_{k\geq 1:\; \lambda_k>1/2} \frac{\Gamma(n+1+\omega^k)}{\Gamma(n+1)}\Xi_k v_k
= \sum_{k=1}^{m-1} \Pi_{n,k}.
\end{equation*}
Lemma \ref{mgconv} implies:
\begin{lem}\label{cov_mat}
 For all $ k\in \{1,\ldots,m-1\}\setminus\{\frac{m}{6},\frac{5m}{6}\}$, as $n \to \infty$, we have
\begin{align}
 \Cov\left(\Pi_{n,k}\right) \sim \frac{1}{|2 \lambda_k -1|} n \cdot v_k v_k^*.\label{cov_mat1}
\end{align}
If $6 \mid m$, then
\begin{align}
 \Cov\left(\Pi_{n,m/6}\right) \sim  n \log(n)\cdot v_{m/6} v_{m/6}^*,\label{cov_mat2}\quad
 \Cov\left(\Pi_{n,5m/6}\right) \sim  n \log(n)\cdot v_{5m/6} v_{5m/6}^*.
\end{align}
\end{lem}
This also determines the covariance matrices $\Sigma_k$ in Proposition \ref{prop1}: We have
\begin{align}\label{sig_k_def1}
\Sigma_k = \frac{1}{|2 \lambda_k -1|} \cdot v_k v_k^* + \frac{1}{|2 \lambda_{m-k} -1|} \cdot v_{m-k} v_{m-k}^*
\end{align}
for $ k\in \{1,\ldots,\lceil m/2\rceil-1\}\setminus\{\frac{m}{6}\}$ as well as
\begin{align}\label{sig_k_def2}
\Sigma_{m/6} &=  v_{m/6} v_{m/6}^* + v_{5m/6} v_{5m/6}^*, \mbox{ if } 6\mid m,\\
\Sigma_{m/2} &=\frac{1}{|2 \lambda_{m/2} -1|} \cdot v_{m/2} v_{m/2}^*, \mbox{ if } 2 \mid m.\label{sig_k_def3}
\end{align}
We also need to control correlations of residuals between different eigenspaces. An explicit calculation implies for all $k,\ell\ge 1$ with $k\neq \ell$ and $\lambda_k,\lambda_\ell >\frac{1}{2}$ that
\begin{equation}\label{mix_est}
 \E\left[\left(M_{n,k} - \Xi_k\right)\left(M_{n,\ell} - \Xi_\ell\right)\right] = \mathrm{O} \left(n^{-1}+n^{\lambda_{k+\ell} - \lambda_k - \lambda_\ell}\right).
\end{equation}
The bound (\ref{mix_est}) implies:
\begin{lem}
Let $k, \ell \geq 1$ with $k \neq \ell$ and $n\to\infty$. If $\lambda_k, \lambda_\ell > \frac{1}{2}$ or $\lambda_k, \lambda_\ell \le \frac{1}{2}$ then
\begin{equation*}
 \Cov\left(\Pi_{n,k}, \Pi_{n,\ell}\right) = o(n).
\end{equation*}
If $\lambda_k>\frac{1}{2}$ and $\lambda_\ell \le\frac{1}{2}$  then
\begin{equation*}
 \Cov\left(\Pi_{n,k}, \Pi_{n,\ell}\right) = 0.
\end{equation*}
\end{lem}
These moments estimates are sufficient to subsequently properly scale the projections of the residuals and to guarantee the finiteness of the Zolotarev metric $\zeta_3$ used.

\subsection{The rank of the covariance matrices}\label{sec:32}
The covariance matrices $\Sigma^{(m)}$  in Theorem \ref{thm1} and \ref{thm2} appear as the sums of the covariance matrices in (\ref{sig_k_def1}) and (\ref{sig_k_def3}) if  $6\nmid m$  and as the covariance matrix in (\ref{sig_k_def2}) if $6\mid m$. We obtain their ranks as follows:
\begin{thm}\label{thm_rank}
For $6 \nmid m$, the matrix
\begin{equation}
\Sigma^{(m)}= \sum_{k=1}^{m-1}\frac{1}{|2\lambda_k-1|} v_k v_k^*
\end{equation}
has rank $m-1$, while for $6 \mid m$,
\begin{equation}
\Sigma^{(m)}= v_{m/6}v_{m/6}^*+ v_{5m/6}v_{5m/6}^*
\end{equation}
has rank two.
\end{thm}
\begin{proof}
Note that the matrix-vector product $mv_k v_k^* x$ is the orthogonal projection of $x\in \C^m$ onto the eigenspace spanned by $v_k$.
Hence, we have
\begin{equation*}
 \mathrm{Id}_m = \sum_{k=0}^{m-1}mv_k v_k^*.
\end{equation*}
The matrix $m\Sigma^{(m)}$ can be interpreted as the orthogonal projection onto  $\mathrm{span}\{v_1, \ldots, v_{m-1}\}$  for the case $6 \nmid m$  and onto the subspace
$\mathrm{span}\{v_{m/6},  v_{5m/6}\}$ for $6 \mid m$. Hence, we obtain the ranks $m-1$ and $2$, respectively.
\end{proof}

\subsection{Embedding into a random binary search tree}\label{sec:33}
In this section we describe the self-similarity of the martingale limits $\Xi_k$ by deriving an almost sure recursive equation for the  $\Xi_k$ and a distributional recurrence for the sequence $(R_n)_{n\ge 0}$ which
extends to a recurrence for the residuals in (\ref{strong}) as well as to the normalized residuals $X_{n,k}$ of the projections of
the $R_n$.

For this, we embed the cyclic urn process into a random binary search tree. The random binary search tree starts with one external node. In each step one of the external nodes is chosen uniformly at random (and independently from the previous choices) and replaced by one internal node with two children, the children being  external nodes attached along a left and right branch. The cyclic urn is embedded into the evolution of the random binary search tree by labeling its external nodes by the types of the balls. The initial external node is labeled by type $0$. Whenever an  external node of type $j\in\{0,\ldots,m-1\}$ is replaced by an internal node its (new) left child gets label $j$, its right child gets label $j+1\mod m$. Note, that the external nodes of the tree correspond to the balls in the urn. A related embedding was exploited in \cite[Section 6.3]{knne14}. Note that the binary search tree starting with one external node labeled $0$ decomposes into its left and right subtree starting with external nodes of types $0$ and $1$, respectively. The size (number of internal nodes) $I_n$ of the left subtree is uniformly distributed on $\{0,\ldots,n-1\}$. This implies, with $J_n:=n-1-I_n$, the recurrence
\begin{align}\label{basic_rec}
R_n^{[0]} = R_{I_n}^{[0],(0)} + R_{J_n}^{[1],(1)}=R_{I_n}^{[0],(0)} + {\cal R}^t R_{J_n}^{[0],(1)},
\end{align}
where the sequences $(R_{n}^{[0],(0)})_{n\ge 0}$ and $(R_{n}^{[1],(1)})_{n\ge 0}$ denote the composition vectors of the cyclic urns given by the evolutions of the left and right subtrees of the root of the binary search tree (upper indices $(0)$ and $(1)$  denoting left and right subtree, upper indices $[0]$ and $[1]$  denoting the initial type). They are independent and independent of $I_n$. Note that the second equation in (\ref{basic_rec}) is due to (\ref{shift}) where the $R_{n}^{[0],(1)}$ are chosen appropriately for pointwise equality. Now, applying the  transformation and scaling which turns $R_n$ into  $M_{n,k}$ to the left and right hand side of (\ref{basic_rec}), letting $n\to\infty$ and using the convergence in (\ref{mg_conv}) implies the following recursive equation for the  $\Xi_k$:
\begin{prop}
For all $k\ge 1$ with $\lambda_k>\frac{1}{2}$ there exist independent random variables $U$, $\Xi_k^{(0)}$, $\Xi_k^{(1)}$ such that
\begin{align}\label{rec_xik}
\Xi_k = U^{\omega^k} \Xi_k^{(0)} + \omega^k (1-U)^{\omega^k} \Xi_k^{(1)} + g_k(U),
\end{align}
where
\begin{equation*}
 g_k(u):=\frac{1}{\Gamma(1+\omega^k)}\left(u^{\omega^k} + \omega^k(1-u)^{\omega^k} -1 \right)
\end{equation*}
and $U$ has the uniform distribution on $[0,1]$ and $\Xi_k^{(0)}$ and $\Xi_k^{(1)}$ have the same distribution as $\Xi_k$.
\end{prop}
Alternatively, the martingale limits $\Xi_k$ can be written explicitly as deterministic functions of the limit of the random binary search tree when interpreting the evolution of the random binary search tree as a transient Markov chain and its limit as a random variable in the Markov chain's Doob-Martin boundary, see \cite{evgrwa12,gr14}. From this representation the self-similarity relation (\ref{rec_xik}) can be read off as well.

\subsection{Proving convergence}\label{sec:34}
Note that the left and right hand sides of (\ref{basic_rec}) and (\ref{rec_xik}) are linked via the convergence of the $M_{n,k}$ towards $\Xi_k$. This allows to come up with a recurrence for the vector $(X_{n,1},\ldots,X_{n,\lfloor m/2\rfloor})$ in Proposition \ref{prop1}. The reader is asked to trust the authors that the techniques developed in \cite{ne14} for a univariate problem can be extended to the multivariate recurrences for $(X_{n,1},\ldots,X_{n,\lfloor m/2\rfloor})$ and that the same type of proof as in \cite{ne14} based on the Zolotarev metric $\zeta_3$ can be applied.

\subsection*{Acknowledgement} We thank Johannes Brahms (op.~120) for inspiration  while doing research on the subject of this paper.

\end{document}